\documentclass[a4paper,11pt]{article}
\usepackage[utf8x]{inputenc}
\usepackage{amssymb}
\usepackage{amsthm}
\usepackage{amsmath}
\usepackage{enumerate}
\usepackage{verbatim}
\usepackage{color}
\usepackage[left=2cm,right=2cm,top=2cm,bottom=2cm,includefoot]{geometry}
\title{}
\author{}
\newtheorem{theorem}{Theorem}
\newtheorem{lemma}[theorem]{Lemma}

\newtheorem{corollary}[theorem]{Corollary}
\newtheorem{definition}{Definition}

\newcommand{\igood}[1]{\textbf{(s#1)}}
\newcommand{\icanonical}[1]{\textbf{(w#1)}}

\newcommand{\GG}{\mathcal{G}}
\newcommand{\PP}{\mathcal{P}}

\title{Strong immersions and maximum degree}
\author{Zden\v ek Dvo\v r{\'a}k\thanks{Computer Science Institute of Charles University, Prague, Czech Republic.
E-mail: \texttt{rakdver@iuuk.mff.cuni.cz}.  Supported the Center of Excellence -- Inst. for Theor. Comp. Sci., Prague (project P202/12/G061 of Czech Science Foundation), and
by project LH12095 (New combinatorial algorithms - decompositions, parameterization, efficient solutions) of Czech Ministry of Education.
}\and Tereza Klimo\v{s}ov{\'a}\thanks{Institute of Mathematics and DIMAP, University of Warwick, Coventry, UK. E-mail: \texttt{T.Klimosova@warwick.ac.uk}. Her~work leading to this
invention has received funding from the European Research Council under the European Union's Seventh Framework Programme (FP7/2007-2013)/ERC grant
agreement no.~259385.}}
\begin{document}

\maketitle

\begin{abstract}
A graph $H$ is \emph{strongly immersed} in $G$ if $G$ is obtained from $H$ by a sequence of vertex splittings (i.e., lifting some pairs
of incident edges and removing the vertex) and edge removals.
Equivalently, vertices of $H$ are mapped to distinct vertices of $G$ (\emph{branch vertices}) and edges of $H$ are mapped to pairwise edge-disjoint
paths in $G$, each of them joining the branch vertices corresponding to the ends of the edge and not containing any other branch vertices.
We show that there exists a function $d\colon N\to N$ such that for all graphs $H$ and $G$, if $G$ contains a strong immersion of the star $K_{1,d(\Delta(H))|V(H)|}$ whose branch vertices
are $\Delta(H)$-edge-connected to one another, then $H$ is strongly immersed in $G$.  This has a number of structural consequences for graphs
avoiding a strong immersion of $H$.  In particular, a class $\GG$ of simple $4$-edge-connected graphs contains all graphs of maximum degree $4$ as strong immersions
if and only if $\GG$ has either unbounded maximum degree or unbounded tree-width.
\end{abstract}

In this paper, graphs are allowed to have parallel edges and loops, where each loop contributes $2$ to the degree of the incident vertex.  A graph without parallel edges and loops is called \emph{simple}.

Various containment relations have been studied in structural graph theory.  The best known ones are \emph{minors} and \emph{topological minors}.
A graph $H$ is a \emph{minor} of $G$ if it can be obtained from $G$ by a sequence of edge and vertex removals and edge contractions.
A graph $H$ is a \emph{topological minor} of $G$ if a subdivision of $H$ is a subgraph of $G$, or equivalently, if $H$ can be obtained from $G$
by a sequence of edge and vertex removals and by suppressions of vertices of degree two.  In a fundamental series of papers, Robertson and Seymour developed
the theory of graphs avoiding a fixed minor, giving a description of their structure~\cite{robertson2003graph} and proving that every proper minor-closed class
of graphs is characterized by a finite set of forbidden minors~\cite{rs20}.  The topological minor relation is somewhat harder to deal with (and in particular,
there exist proper topological minor-closed classes that are not characterized by a finite set of forbidden topological minors), but a description of their structure
is also available~\cite{gmarx,topmin}.

In this paper, we consider a related notion of a graph immersion. Let $H$ and $G$ be graphs.  
An \emph{immersion} of $H$ in $G$ is a function $\theta$ from vertices and edges of $H$ such that
\begin{itemize}
\item $\theta(v)$ is a vertex of $G$ for each $v\in V(H)$, and $\theta\restriction V(H)$ is injective.
\item $\theta(e)$ is a connected subgraph of $G$ for each $e\in E(H)$, and if $f\in E(H)$ is distinct from $e$, then
$\theta(e)$ and $\theta(f)$ are edge-disjoint.
\item If $e\in E(H)$ is incident with $v\in V(H)$, then $\theta(v)$ is a vertex of $\theta(e)$,
and if $e$ is a loop, then $\theta(e)$ contains a cycle passing through $\theta(v)$.
\end{itemize}
An immersion $\theta$ is \emph{strong} if it additionally satisfies the following condition:
\begin{itemize}
\item If $e\in E(H)$ is not incident with $v\in V(H)$, then $\theta(e)$ does not contain $\theta(v)$.
\end{itemize}
When we want to emphasize that an immersion does not have to be strong, we call it \emph{weak}.
Let $E(\theta)$ denote $\bigcup_{e\in E(H)} E(\theta(e))$ and let $V(\theta)$ denote $\bigcup_{e\in E(H)} V(\theta(e))$.
Let us note that by choosing the subgraphs $\theta(e)$ as small as possible, we can assume that
$\theta(e)$ is a path with endvertices $\theta(u)$ and $\theta(v)$
if $e$ is a non-loop edge of $H$ joining $u$ and $v$, and that $\theta(e)$ is a cycle containing $\theta(v)$
if $e$ is a loop of $H$ incident with $v$; we call an immersion satisfying these constraints \emph{slim}.

If $H$ is a topological minor of $G$, then $H$ is also strongly immersed in $G$.  On the other hand,
an appearance of $H$ as a minor does not imply an immersion of $H$, and conversely, an appearance of $H$ as a strong immersion
does not imply the appearance as a minor or a topological minor.  Nevertheless, many of the results for minors and topological minors
have analogues for immersions and strong immersions.  For example, any simple graph with minimum degree at least $200k$ contains
a strong immersion of the complete graph $K_k$ (DeVos et al.~\cite{mindim}), as compared to similar sufficient minimum degree conditions
for minors ($\Omega(k\sqrt{\log k})$, Kostochka~\cite{kostomindeg}, Thomason~\cite{thommindeg}) and topological minors ($\Omega(k^2)$,
Bollob\'as and Thomason~\cite{BoTh}, Koml\'os and Szemer\'edi \cite{KoSz}). A structure theorem for weak immersions appears
in DeVos et al.~\cite{mattpriv} and Wollan~\cite{wollims}.  Furthermore, every proper class of graphs closed on weak immersions
is characterized by a finite set of forbidden immersions~\cite{rs23}.

Chudnovsky et al.~\cite{seymgrid} proved the following variation
on the grid theorem of Robertson and Seymour~\cite{RSey}.
\begin{theorem}\label{thm-grid}
For every $g\ge 1$, there exists $t\ge 0$ such that every $4$-edge-connected graph of tree-width at least $t$ contains the $g\times g$ grid
as a strong immersion.
\end{theorem}
The variant of Theorem~\ref{thm-grid} for weak immersions was also proved by Wollan~\cite{wollims}.
Note that unlike the grid theorem for minors~\cite{RSey}, Theorem~\ref{thm-grid} does not admit a weak converse---there exist graphs of bounded
tree-width containing arbitrarily large grids as strong immersions.
A connected graph $H$ with at least three vertices is a \emph{multistar} if it has no loops and contains
a vertex $c$ incident with all its edges.  The vertex $c$ is the \emph{center} of the multistar and all its other vertices are \emph{rays}.
We write $c(H)$ for the center of the multistar and $R(H)$ for the set of its rays.
Let the multistar with $n$ rays of degree $k$ be denoted by $S_{n,k}$.  Note that every graph with at most $n$ vertices and with maximum degree at
most $k$ is contained as a strong immersion in $S_{n,k}$, which has tree-width $1$.  Furthermore, subdividing each edge of $S_{n,k}$ results
in a simple graph of tree-width $2$ containing every graph with at most $n$ vertices and with maximum degree at
most $k$ as a strong immersion.

Consequently, to turn Theorem~\ref{thm-grid} into an approximate characterization, we need to deal with the star-like graphs.
The main result of this paper essentially states that if the maximum degree of $H$ is $k$ and a $k$-edge-connected graph $G$ contains
a sufficiently large star as a strong immersion, then $G$ also contains $H$ as a strong immersion.  Let us now state the result more precisely.

A \emph{$k$-system of magnitude $d$} in a graph $G$ is a pair $(H,\sigma)$, where $H$ is a multistar and $\sigma$ is a strong immersion of $H$ in $G$
satisfying the following conditions:
\begin{description}
\item[\igood{1}] $H$ has at least $d$ edges,
\item[\igood{2}] rays of $H$ have degree at most $k$ (in $H$), and
\item[\igood{3}] for each $v\in R(H)$, there exists no edge cut in $G$ of size less than $k$ separating 
$\sigma(c(H))$ from $\sigma(v)$.
\end{description}
Let us remark that by Menger's theorem, \igood{3} implies that $G$ contains $k$ pairwise edge-disjoint paths
from $\sigma(c(H))$ to $\sigma(v)$ (not necessarily belonging to or disjoint with the immersion).

A strong immersion $\theta$ of $S_{n,k}$ in a graph $G$ \emph{respects} a strong immersion $\sigma$ of a multistar $H$ in $G$ if
$\theta(c(S_{n,k}))=\sigma(c(H))$ and $\theta(R(S_{n,k}))\subseteq \sigma(R(H))$.
Let us define $d(k)=(2k+1)^{8k+4}k^2(k+1)$.

\begin{theorem}\label{thm-main}
If $k\ge 3$ and $n\ge 2$ are integers and $(H,\sigma)$ is a $k$-system of magnitude
at least $d(k)n$ in a graph $G$, then $G$ contains $S_{n,k}$ as a strong immersion respecting $(H,\sigma)$.
\end{theorem}

If $G$ is $k$-edge-connected and contains a vertex $c$ with at least $d$ distinct neighbors,
then the neighborhood of $c$ contains a $k$-system of magnitude $d$.  Let us recall that if a graph $F$ has $n$ vertices
and maximum degree at most $k$, then $F$ is strongly immersed in $S_{n,k}$.

Furthermore, the relation of strong immersion is transitive.
If $H_1$ has an immersion $\theta_1$ in $H$ and $H$ has an immersion $\theta$ in $G$, then let
$\theta\circ\theta_1$ be defined as follows: $(\theta\circ\theta_1)(v)=\theta(\theta_1(v))$ for each $v\in V(H_1)$
and $(\theta\circ\theta_1)(e)=\bigcup_{f\in E(\theta_1(e))} \theta(f)$ for each $e\in E(H_1)$.
Note that $\theta\circ\theta_1$ is an immersion of $H_1$ in $G$, and if $\theta_1$ and $\theta$ are strong,
then $\theta\circ\theta_1$ is strong.  Consequently, Theorem~\ref{thm-main} has the following corollary.

\begin{corollary}\label{cor-main}
For every integer $k\ge 3$ and a graph $F$ of maximum degree at most $k$, if a $k$-edge-connected graph $G$ contains a vertex
with at least $d(k)|V(F)|$ distinct neighbors, then $F$ appears in $G$ as a strong immersion.
\end{corollary}

The version of Corollary~\ref{cor-main} for weak immersions was previously obtained by a different method by Marx and Wollan~\cite{weakdeg}.
As a consequence of Corollary~\ref{cor-main}, we obtain the following strengthening of Theorem~\ref{thm-grid}.

\begin{theorem}\label{thm-class}
Let $\GG$ be a class of $4$-edge-connected simple graphs.  The following propositions are equivalent:
\begin{itemize}
\item[\textrm{(i)}] There exists a graph $F$ of maximum degree $4$ that does not appear as a weak immersion in any graph in $\GG$.
\item[\textrm{(ii)}] There exists a graph $F$ of maximum degree $4$ that does not appear as a strong immersion in any graph in $\GG$.
\item[\textrm{(iii)}] There exists an integer $t\ge 4$ such that every graph in $\GG$ has tree-width at most $t$ and maximum degree at most $t$.
\end{itemize}
\end{theorem}
Let us remark that the assumption that graphs in $\GG$ are simple is important---if $\GG$ is the class of all graphs that can be obtained from paths by replacing each edge by
at least four parallel edges, then $\GG$ satisfies (ii), but not (i) and (iii).  Furthermore, if $\GG$ is the class of graphs obtained from
simple $4$-edge-connected $4$-regular graphs of bounded tree-width (say at most $10$) by replacing one edge by any positive number of parallel edges, then
$\GG$ satisfies (i) and (ii), but not (iii).  The implications $\mathrm{(iii)}\Rightarrow \mathrm{(i)}\Rightarrow\mathrm{(ii)}$ hold even
if $\GG$ contains non-simple graphs, though.  Obviously, in propositions (i) and (ii), we could restrict $F$ to be a square grid.

Flows in networks can be used to determine whether a $k$-system of large magnitude with a given center and rays exists.  This enables us
to restate Theorem~\ref{thm-main} in the following form.

\begin{theorem}\label{thm-sep}
Let $G$ be a graph and $c$ a vertex of $G$.
Let $X\subseteq V(G)\setminus \{c\}$ be any set of vertices such that $G$ contains no edge cut of size less than $k$ separating $c$ from a vertex in $X$.
If a graph $F$ of maximum degree at most $k\ge 3$ does not appear in a graph $G$ as a strong immersion, then
there exist sets $Y\subseteq X$ and $K\subseteq E(G)$ such that $k|Y|+|K|<d(k)|V(F)|$ and the component of $G-Y-K$ that
contains $c$ does not contain any vertex of $X$.
\end{theorem}

Theorem~\ref{thm-sep} forms a basis for a structure theorem for strong immersions analogous to the one for weak immersions~\cite{mattpriv,wollims},
which we develop in a future paper.  Here, let us state just the first step towards this structure.
\begin{theorem}\label{thm-str}
For every graph $F$ and an integer $m\ge 0$, there exists a constant $M$ such that the following holds.  Let $G$ be a graph and $X\subseteq V(G)$ a set of its vertices such that no two
vertices of $X$ are separated by an edge cut of size less than $M$ in $G$.  Let $G_X$ be the graph with vertex set $X$ in that two vertices $u,v\in X$ are adjacent if
$G$ contains $m$ pairwise edge-disjoint paths joining $u$ with $v$ and otherwise disjoint with $X$.  If $G$ does not contain $F$ as a strong immersion, then $G_X$ is connected.
\end{theorem}
Let $K_F$ be a multistar with rays $V(F)$, such that each vertex $v\in V(F)$ has the same degree in $F$ and in $K_F$.
Observe that if $m\ge 2|E(F)|$, then the graph $G_X$ cannot contain the star $K_{1, |V(F)|}$ as a minor, as otherwise Menger's theorem would imply that
$G$ contains $K_F$ as a strong immersion, and consequently that $G$ contains $F$ as a strong immersion.  This restricts the structure of $G_X$ significantly,
and to obtain a structure theorem, it remains to argue how the rest of the graph can attach to this well-structured part of $G$.

In Section~\ref{sec-cor}, we prove Theorems~\ref{thm-class}, \ref{thm-sep} and \ref{thm-str} under assumption that Theorem~\ref{thm-main} holds.  Section~\ref{sec-main} is devoted to the proof of Theorem~\ref{thm-main}.

\section{Corollaries}\label{sec-cor}

\begin{proof}[Proof of Theorem~\ref{thm-class}.]
The implication $\mathrm{(i)}\Rightarrow\mathrm{(ii)}$ is trivial.

Note that for every graph $F$ of maximum degree $4$, there exists an integer $g$ such that $F$ is strongly immersed in the $g\times g$ grid.
Let $t_1$ be the constant of Theorem~\ref{thm-grid} for this $g$.  Let $t_2=d(4)|V(F)|$, and let $t=\max(t_1,t_2)$.
If a graph $F$ of maximum degree $4$ does not appear as a strong immersion in a graph $G\in \GG$, then by Theorem~\ref{thm-grid},
$G$ has tree-width at most $t_1\le t$, and by Corollary~\ref{cor-main}, $G$ has maximum degree at most $t_2\le t$.
Therefore, $\mathrm{(ii)}\Rightarrow\mathrm{(iii)}$ holds.

Suppose now that $\mathrm{(iii)}$ holds, i.e., there exists an integer $t\ge 4$ such that every graph $G\in \GG$ has tree-width at most $t$ and maximum degree at most $t$.
Let $F$ be a sufficiently large $4$-regular expander (see~\cite{expand} for a construction); say, for some $\varepsilon>0$, $|V(F)|\ge 5t+5+\frac{3(t+1)^2}{\varepsilon}$ and for each $S\subseteq V(F)$ of size at most $|V(F)|/2$,
there are at least $\varepsilon|S|$ edges in $F$ between $S$ and $V(F)\setminus S$.  We claim that $F$ does not appear as a weak immersion in $G$.

Suppose on the contrary that $\theta$ is an immersion of $F$ in $G$, and let $T=\theta(V(F))$.  Since $G$ has tree-width at most $t$, there exist sets
$A,B\subset V(G)$ such that $A\cup B=V(G)$, no edge of $G$ joins a vertex in $A\setminus B$ with a vertex in $B\setminus A$,
$|A\cap B|\le t+1$, $|A\cap T|\ge \frac{1}{3}(|T|-2t-2)$ and $|B\cap T|\ge \frac{1}{3}(|T|-2t-2)$.
Without loss of generality, assume that $|T\setminus B|\le |T\setminus A|$.
Let $S_A=\theta^{-1}(T\setminus B)$ and $S_B=\theta^{-1}(T\cap B)=V(F)\setminus S_A$.  Note that $\frac{1}{3}(|V(F)|-5t-5)\le |S_A|\le |V(F)|/2$.
Let $Z$ be the set of edges of $F$ between $S_A$ and $S_B$; we have $|Z|\ge \varepsilon|S_A|$.  Consequently, the subgraph $Q=\bigcup_{e\in Z} \theta(e)$
contains at least $\varepsilon|S_A|$ edges incident with vertices of $A\cap B$, and at least one vertex of $A\cap B$ has degree at least
$\varepsilon|S_A|/(t+1)\ge t+1$.  This contradicts the assumption that the maximum degree of $G$ is at most $t$, showing that $\mathrm{(iii)}\Rightarrow\mathrm{(i)}$ holds.
\end{proof}

\begin{proof}[Proof of Theorem~\ref{thm-sep}.]
Since $G$ does not contain $F$ as a strong immersion, Theorem~\ref{thm-main} implies that $G$ does not contain a $k$-system of magnitude $d(k)|V(F)|$.
Let $G'$ be the network with the vertex set $V(G)\cup \{z\}$, where $z$ is a new vertex not appearing in $V(G)$, and the edge set defined as follows: For each edge $e\in E(G)$ not incident with a vertex in $X$, add a pair
of edges in opposite directions joining the endvertices of $e$.  For each edge $e\in E(G)$ joining a vertex $u\not\in X$ with a vertex $v\in X$,
add an edge directed from $u$ to $v$.  For each vertex $x\in X$, add an edge directed from $x$ to $z$.  The edges incident with $z$ have capacity $k$,
while all other edges of the network have capacity $1$.
If $G'$ contained a flow of size $d(k)|V(F)|$ from $c$ to $z$, then the corresponding pairwise edge-disjoint paths in $G$ would form a $k$-system of magnitude
$d(k)|V(F)|$, as each vertex of $X$ is contained in at most $k$ such paths.  Consequently, no such flow exists.

By the flow-cut duality, it follows that $G'$ contains an edge cut $K'$ of capacity less than $d(k)|V(F)|$ separating $c$ from $z$.  Let $Y$ be the set of vertices $y\in X$
such that the edge $xz$ belongs to $K'$.  Let $K$ be the set of edges of $G$ corresponding to the edges of $K'$ not incident with $z$.
Clearly, $G-Y-K$ contains no path from $c$ to a vertex of $X$.  Furthermore, $k|Y|+|K|$ is equal to the capacity of $K'$, and thus $k|Y|+|K|<d(k)|V(F)|$.
\end{proof}

\begin{proof}[Proof of Theorem~\ref{thm-str}.]
Let $k=\max(\Delta(F),3)$, $s=d(k)|V(F)|$ and $M=ms^3+s^2$.  Suppose that a graph $G$ and a set $X\subseteq V(G)$ satisfy the assumptions of Theorem~\ref{thm-str}.
Consider any nonempty disjoint sets $A,B\subset X$ such that $A\cup B=X$.  Let $c_A$ be an arbitrary vertex of $A$ and apply Theorem~\ref{thm-sep} for $c_A$ and $B$,
obtaining sets $Y_B\subseteq B$ and $K_B\subseteq E(G)$, where $k|Y_B|+|K_B|<s$, such that the component of $G-K_B-Y_B$ that contains $c_A$ does not contain any vertex of $B$.
For each $y\in Y_B$, apply Theorem~\ref{thm-sep} for $y$ and $A$, obtaining sets $Y^y_A\subseteq A$ and $K^y_A\subseteq E(G)$, where $k|Y^y_A|+|K^y_A|<s$, such that the component of $G-K^y_A-Y^y_A$ that contains $y$ does not contain any vertex of $A$.  Let $K=K_B\cup\bigcup_{y\in Y_B} K^y_A$
and let $Y_A=\bigcup_{y\in Y_B} Y^y_A$, and note that $|K|\le s^2$ and $|Y_A|\le s^2$.

Let $c_B$ be an arbitrary vertex of $B$. By Menger's theorem, there exists a set $\PP_0$ of $M$ pairwise edge-disjoint paths from $c_A$ to $c_B$ in $G$.
Let $\PP\subseteq \PP_0$ consist of the paths that do not contain edges of $K$; we have $|\PP|\ge M-s^2$. 
Consider a path $P\in \PP$.  Let $v_0$, $v_1$, \ldots, $v_p$ be the vertices of $P$ in order, where $v_0=c_A$ and $v_p=c_B$.
Let $j>0$ be the smallest index such that $v_j$ belongs to $B$.  As the component of $G-K_B-Y_B$ that contains $c_A$ does not contain any vertex of $B$,
the vertex $v_j$ belongs to $Y_B$.  Let $i$ be the largest index such that $i<j$ and $v_i$ belongs to $A$.  As the component
of $G-K^{v_j}_A-Y^{v_j}_A$ that contains $v_j$ does not contain any vertex of $A$, it follows that $v_i$ belongs to $Y^{v_j}_A\subseteq Y_A$.
Consequently, $G$ contains a set of $|\PP|$ pairwise edge-disjoint paths
joining vertices of $Y_A$ with vertices of $Y_B$ and otherwise disjoint from $X$.  By the pigeonhole principle, there exist vertices $a\in A$ and $b\in B$
incident with at least $\frac{|\PP|}{|Y_A||Y_B|}\ge m$ of these paths, and thus $ab$ is an edge of $G_X$.

Therefore, for all nonempty disjoint sets $A,B\subset X=V(G_X)$ such that $A\cup B=V(G_X)$, there exists an edge between a vertex of $A$ and a vertex of $B$ in $G_X$.
It follows that $G_X$ is connected.
\end{proof}

\section{Proof of Theorem~\ref{thm-main}}\label{sec-main}

We need the following variation on the Mader's splitting theorem~\cite{mader}.
Let $G$ be a graph, $x$ a vertex of $G$ and for each $s,t\in V(G)\setminus\{x\}$,
let $\lambda(s,t)$ denote the maximum number of pairwise edge-disjoint paths
between $s$ and $t$ in $G$.  Let $e$ and $f$ be edges joining $x$ to vertices $u$ and $v$, respectively, and let $G'$ be the graph
obtained from $G-\{e,f\}$ by adding a new edge joining $u$ with $v$.
We say that the pair of edges $e$ and $f$ is \emph{splittable} if for every $s,t\in V(G')\setminus\{x\}$,
the graph $G'$ contains $\lambda(s,t)$ pairwise edge-disjoint paths
between $s$ and $t$.  We say that $G'$ is obtained by \emph{lifting} the edges $e$ and $f$.  Note that $G'$ is immersed
in $G$.

\begin{theorem}[Frank~\cite{frank}]\label{split}
Let $G$ be a graph and let $x$ be a vertex of $G$ not incident with any $1$-edge cut.
If $x$ has degree $m\neq 3$,
then there are $\lfloor m/2\rfloor$ pairwise disjoint splittable pairs of edges incident with $x$.
\end{theorem}

If $(H,\sigma)$ is a $k$-system and $\sigma$ is slim, we consider the paths in $\sigma(E(H))$ to be directed away from
$\sigma(c)$, where $c=c(H)$.  That is, if $e=cv$ is an edge of $H$, then $\sigma(c)$ is the first vertex
of $\sigma(e)$ and $\sigma(v)$ is the last vertex of $\sigma(e)$.

\begin{definition}\label{canonical:def}
We say that a triple $(G,H,\sigma)$ is \emph{well-behaved} if $(H,\sigma)$ is a $k$-system of magnitude $d$ in a graph $G$,
such that $\sigma$ is slim and the following conditions are satisfied:
\begin{description}
 \item[\icanonical{1}] $G$ is 3-edge connected,
 \item[\icanonical{2}] vertices in $V(G)\setminus V(\sigma)$ have degree exactly $3$,
 \item[\icanonical{3}] for every $v\in R(H)$, if $K$ is an $k$-edge cut in $G$ separating $\sigma(c(H))$ from $\sigma(v)$,
    then $K$ consists of the edges incident with $\sigma(v)$,
 \item[\icanonical{4}] every edge of $E(G)\setminus E(\sigma)$ is incident with a vertex in $\sigma(R(H))$ of degree exactly $k$,
 \item[\icanonical{5}] for each vertex $v\in V(\sigma)\setminus \sigma(V(H))$, at most one edge incident with $v$
   does not belong to $E(\sigma)$.  Furthermore, if there is such an edge and $v$ belongs to the path $\sigma(e)$ for an edge $e\in E(H)$,
   then $v$ is the next-to-last vertex of $\sigma(e)$ and the last vertex of $\sigma(e)$ has degree exactly $k$.
\end{description}
\end{definition}

Let $H$ be a multistar with a strong immersion $\sigma$ in a graph $G$ and let $H'$ be a multistar with a strong immersion $\sigma'$ in a graph $G'$.
We say that $(G',H',\sigma')$ is a \emph{reduction} of $(G,H,\sigma)$ if there exists a weak immersion $\theta$ of $G'$
in $G$ satisfying the following conditions:
\begin{itemize}
\item $\theta(\sigma'(c(H')))=\sigma(c(H))$,
\item $\theta(\sigma'(R(H')))\subseteq\sigma(R(H))$, and
\item if $e\in E(G')$ is not incident with a vertex $v\in \sigma'(V(H'))$, then $\theta(e)$ does not contain $\theta(v)$.
\end{itemize}
Note that if a strong immersion $\alpha$ of $S_{n,k}$ in $G'$ respects $(H',\sigma')$, then $\theta\circ\alpha$ is
a strong immersion of $S_{n,k}$ in $G$ respecting $(H,\sigma)$.  Furthermore, the reduction relation is transitive.

\begin{lemma}\label{canonical}
If $(H_0,\sigma_0)$ is a $k$-system of magnitude $d$ in a graph $G_0$, then there exists a reduction
$(G, H,\sigma)$ of $(G_0,H_0,\sigma_0)$ such that $(H,\sigma)$ is a $k$-system of magnitude $d$
and $(G,H,\sigma)$ is well-behaved.
\end{lemma}
\begin{proof}
Let $G$ with a $k$-system $(H,\sigma)$ of magnitude $d$ be chosen so that $(G, H,\sigma)$ is a reduction of $(G_0,H_0,\sigma_0)$,
and subject to that $|V(G)|+|E(G)|+|E(\sigma)|$ is minimal.  Such triple $(G,H,\sigma)$ exists, since $(G_0,H_0,\sigma_0)$ is a reduction of itself.
Clearly, $\sigma$ is slim and $G$ has no loops.  We claim that $(G,H,\sigma)$ is well-behaved.  Let us discuss the conditions \icanonical{1}, \ldots, \icanonical{5}
separately.

\begin{description}
 \item[\icanonical{1}] Suppose that there is an edge cut $K$ of size at most $2$ in $G$;
   we can assume that $K$ is minimal.
   Since $k\ge 3$, \igood{3} implies that all vertices of $\sigma(V(H))$ are in the same component $C$ of $G-K$.
   Let $G'=C$ if $|K|\le 1$; let $G'$ be the graph obtained
   from $C$ by adding an edge $e$ (possibly a loop) between the vertices in $C$ incident with $K$
   if $|K|=2$.  Let $\sigma'(v)=\sigma(v)$ for $v\in V(H)$.  Consider $f\in E(H)$.  If $\sigma(f)$ contains two edges of $K$,
   then let $\sigma'(f)=(\sigma(f)\cap C)+e$, otherwise let $\sigma'(f)=\sigma(f)\cap C$.
   Observe that $(G',H,\sigma')$ is a reduction of $(G,H,\sigma)$.  This is a contradiction, as $(G, H,\sigma)$ was chosen
   so that $|V(G)|+|E(G)|+|E(\sigma)|$ is minimal.

\item[\icanonical{2}] Suppose $G$ satisfies \icanonical{1} and contains a
  vertex $v\in V(G)\setminus V(\sigma)$ of degree greater than $3$. By
  Theorem~\ref{split}, we can lift a pair of edges incident to $v$ without violating the condition \igood{3}.
  This way, we obtain a reduction $(G',H,\sigma')$ contradicting the minimality of $(G,H,\sigma)$.

\item[\icanonical{3}] Suppose $G$ contains an edge cut $K$ of size $k$ separating $\sigma(c(H))$ from $\sigma(v)$ for some $v\in R(H)$,
  such that $K$ does not consist of the edges incident with $\sigma(v)$.
  By \igood{3}, $G-K$ has only two components.  Let $G'$ be the graph obtained from $G$ by
  replacing the component $C$ of $G-K$ that contains $\theta(v)$ by a single
  vertex $w$ of degree $k$, incident with the edges of $K$.  Let $Z\subseteq E(H)$ be the set of edges $e\in E(H)$ such that $\sigma(e)$ contains an edge of $K$.
  Let $H'$ be the multistar obtained from $H$ by making all edges of $Z$ incident with $v$ instead of their original incident ray
  and removing all resulting isolated vertices.  Let $\sigma'$ be the strong immersion of $H'$ in $G'$ such that
  $\sigma'(v)=w$, $\sigma'(x)=\sigma(x)$ for $x\in V(H')\setminus\{v\}$,
  $\sigma'(e)=\sigma(e)$ for $e\in E(H)\setminus Z$, and $\sigma'(e)$ is the segment of the path $\sigma(e)$ between $\sigma'(c(H'))$
  and the first edge of $K$ appearing in the path (inclusive).  Note that $w$ has degree exactly $k$ in $H'$, and thus $H'$ satisfies \igood{2}.
  Furthermore, if $x\in R(H')$, then  $\sigma'(x)$ is not separated from $\sigma(c(H'))$ by an edge cut $K'$ of size less than $k$,
  as otherwise $K'$ (with the edges incident with $w$ replaced by the corresponding edges of $K$) separates $\sigma(x)$ from
  $\sigma(c(H))$; hence, \igood{3} holds for $(H',\sigma')$.
  Also, by \igood{3} for $(H,\sigma)$, there exist $k$ pairwise edge-disjoint paths in $G$ joining $\sigma(v)$ with the edges
  of $K$.  Consequently, there exists a strong immersion of $G'$ in $G$, showing that
  $(G',H',\sigma')$ is a reduction of $(G,H,\sigma)$.  This contradicts the minimality of the latter.
 
\item[\icanonical{4}] Suppose that $G$ satisfies \icanonical{3} and $e\in E(G)\setminus E(\sigma)$ 
is not incident with a vertex of $\sigma(R(H))$ of degree exactly $k$.  Consequently,
$e$ is not contained in any $k$-edge cut separating $\sigma(c(H))$ from a vertex in $\sigma(R(H))$,
and we conclude that $(G-e,H,\sigma)$ is a reduction contradicting the minimality of $(G,H,\sigma)$.

\item[\icanonical{5}] Suppose that $G$ satisfies \icanonical{1} and \icanonical{4} and consider a vertex
$v\in V(\sigma)\setminus \sigma(V(H))$.  If at least two edges incident with $v$
do not belong to $E(\sigma)$, then by Theorem~\ref{split}, there exists a splittable pair
of edges $e$ and $f$ incident with $v$ such that $e\not\in E(\sigma)$.  Let $u$ and $w$ be the endvertices of $e$ and $f$,
respectively, distinct from $v$.  Let $e'$ be an edge incident with $v$ and
not belonging to $\{e\}\cup E(\sigma)$, and let $z$ be the vertex incident with $e'$ distinct from $v$.
By \icanonical{4}, $u$ and $z$ are vertices of $\sigma(R(H))$ of degree exactly $k$.  Let $G'$ be the graph obtained from $G$ by
lifting the edges $e$ and $f$, creating a new edge $h$.  If $f\not\in E(\sigma)$, then let
$H'=H$ and $\sigma'=\sigma$.  Otherwise, consider the edge $f_0\in E(H)$ such that $f\in E(\sigma(f_0))$.
If $w$ appears before $v$ in $\sigma(f_0)$, then let $H'$ be the graph obtained from $H$ by making $f_0$
incident with $\sigma^{-1}(u)$ instead of its original incident ray (and possibly removing the resulting isolated vertex)
and let $\sigma'$ be obtained from $\sigma$ by letting $\sigma'(f_0)$ consist of $h$ and the subpath of $\sigma(f_0)$ between $\sigma(c(H))$
and $w$.  If $w$ appears after $v$ in $\sigma(f_0)$, then let $H'$ be the graph obtained from $H$ by making $f_0$
incident with $\sigma^{-1}(z)$ instead of its original incident ray (and possibly removing the resulting isolated vertex)
and let $\sigma'$ be obtained from $\sigma$ by letting $\sigma'(f_0)$ consist of $e'$ and the subpath of $\sigma(f_0)$ between $\sigma(c(H))$
and $v$.  Note that $(G',H',\sigma')$ is a reduction contradicting the minimality of $(G,H,\sigma)$.

Let us now consider the case that $v\in V(\sigma)\setminus \sigma(V(H))$ is incident with exactly one edge $e$ not belonging to $E(\sigma)$,
where $e$ joins $v$ with a vertex $u$.  Note that $u$ belongs to $\sigma(R(H))$ and has degree exactly $k$.
Let $f_0$ be an edge of $H$ such that $\sigma(f_0)$ contains $v$.
Let $H'$ be obtained from $H$ by making $f_0$
incident with $\sigma^{-1}(u)$ instead of its original incident ray (and possibly removing the resulting isolated vertex)
and let $\sigma'$ be obtained from $\sigma$ by letting $\sigma'(f_0)$ consist of $e$ and the subpath of $\sigma(f_0)$ between $\sigma(c(H))$
and $v$.  Note that $(G,H',\sigma')$ is a reduction of $(G,H,\sigma)$.
By the minimality of $(G,H,\sigma)$, we have that $v$ is the next-to-last vertex of the path $\sigma(f_0)$.
Furthermore, if the last vertex $x$ of $\sigma(f_0)$ had degree greater than $k$, then we could find a reduction of
$(G,H',\sigma')$ contradicting the minimality of $(G,H,\sigma)$ in the same way as in the proof of
\icanonical{3} or \icanonical{4}.
\end{description}
\end{proof}

\begin{lemma}\label{reddeg}
If $(H,\sigma)$ is a $k$-system of magnitude $d$ in a graph $G$, then there exists a
reduction $(G', H',\sigma')$ of $(G,H,\sigma)$ such that $(H',\sigma')$ is a $k$-system of magnitude $\frac{d}{k(k+1)}$ in $G'$,
$(G',H',\sigma')$ is well-behaved and each vertex of $\sigma'(R(H'))$ has degree exactly $k$ in $G'$.
\end{lemma}
\begin{proof}
By Lemma~\ref{canonical}, we can assume that $(G,H,\sigma)$ is well-behaved.  Let $S$ be the set of
all vertices $s\in R(H)$ such that $\sigma(s)$ has degree exactly $k$ in $G$ and let $B=R(H)\setminus S$.  Since $\sigma$ is slim,
\igood{2} and \igood{3} imply that for each $v\in B$, there exists an edge $e\in E(G)\setminus E(\sigma)$
incident with $\sigma(v)$.  By \icanonical{4}, $e$ is incident with a vertex in $\sigma(S)$.
We conclude that $|B|\le k|S|$.  Since $|B|+|S|=|R(H)|$, we have $|S|\ge |R(H)|/(k+1)$.
Let $H_0=H-B$.  Since $H$ is connected and satisfies \igood{2}, we have $|E(H_0)|\ge |S|\ge |R(H)|/(k+1)\ge \frac{d}{k(k+1)}$.  Let
$\sigma_0=\sigma\restriction(V(H_0)\cup E(H_0))$.  Clearly, $(H_0,\sigma_0)$ is a $k$-system of magnitude $\frac{d}{k(k+1)}$ in $G$,
$(G,H_0,\sigma_0)$ is a reduction of $(G,H,\sigma)$ and each vertex of $\sigma_0(R(H_0))$ has degree exactly $k$ in $G$.
Finally, we obtain a well-behaved reduction $(G',H',\sigma')$ by Lemma~\ref{canonical}, since no vertices of degree greater than $k$
belonging to $\sigma'(R(H'))$ are created in its proof.
\end{proof}

Let $(G,H,\sigma)$ be well-behaved, where $(H,\sigma)$ is a $k$-system of magnitude $d$ in a graph $G$ such that each vertex of $\sigma(R(H))$ has degree
exactly $k$.  We can assume that all edges of $G$ between $\sigma(c)$ and $\sigma(R(H))$ belong to $E(\sigma)$,
as otherwise we can add more edges to $H$.  Let $N(\sigma)$ consist of $\sigma(R(H))$ and of all vertices incident with edges
of $E(G)\setminus E(\sigma)$.  Let $M(\sigma)=N(\sigma)\cap V(\sigma)\setminus \sigma(R(H))$.
Let us note that by \icanonical{5}, $M(\sigma)$ is an independent set in $G$.
Consequently, $\sigma(e)$ intersects $N(\sigma)$ in at most two vertices for each $e\in E(H)$.
Let $G'$ be the graph with vertex set $\sigma(c(H))\cup N(\sigma)$ and the edge set defined as follows: the subgraphs of $G$ and $G'$ induced by $N(\sigma)$
are equal; and, the edges incident with $\sigma(c(H))$ are $\{f_e:e\in E(H)\}$, where $f_e$ joins $\sigma(c(H))$ with the first vertex
of $\sigma(e)$ that belongs to $N(\sigma)$.  Let us define a strong immersion $\sigma'$ of $H$ in $G'$ as follows: For $v\in V(H)$, we set
$\sigma'(v)=\sigma(v)$.  For $e\in E(H)$, let $\sigma'(e)$ consist of $f_e$ and of $\sigma(e)\cap G[N(\sigma)]$.
Note that $\sigma'(e)$ has length at most two.  Clearly, $(G',H,\sigma')$ is a reduction of $(G,H,\sigma)$.
We say that $(G',H,\sigma')$ is a \emph{core} of $(G,H,\sigma)$.

Let us define a function $g:E(G')\to E(G)$ as follows:  if $f\in E(G')$ is not incident with $\sigma(c(H))$, then
let $g(f)=f$.  Otherwise, $f=f_e$ for some $e\in E(H)$, and we let $g(f)$ be equal to the last edge of $\sigma(e)$ that does not
belong to $G[N(\sigma)]$.  We say that $g$ is the \emph{origin function} of the core.

\begin{lemma}\label{paspl}
Let $(G,H,\sigma)$ be well-behaved, where $(H,\sigma)$ is a $k$-system of magnitude $d$ in a graph $G$ such that each vertex of $\sigma(R(H))$ has degree
exactly $k$.  If $(G', H, \sigma')$ is the core of $(G,H,\sigma)$, then $(H,\sigma')$ is a $k$-system of magnitude $d$ in $G'$.
\end{lemma}
\begin{proof}
It suffices to check that $(G',H,\sigma')$ satisfies the condition \igood{3}.  Let $g$ be the origin function of $(G', H, \sigma')$.
Consider an edge cut $K$ in $G'$ separating $\sigma'(c(H))$ from $\sigma'(v)$ for some $v\in R(H)$.
Observe that $g(K)$ is an edge cut in $G$ separating $\sigma(c(H))$ from $\sigma(v)$, and thus
$|K|\ge |g(K)|\ge k$ by \igood{3} for $(G,H,\sigma)$.
\end{proof}

Let $(H,\sigma)$ be a $k$-system of magnitude $d$ in a graph $G$.
We say that $(G,H,\sigma)$ is \emph{peeled} if it is well-behaved, each vertex of $\sigma(R(H))$ has degree exactly $k$,
all edges incident with $\sigma(c(H))$ belong to $E(\sigma)$ and $V(G)=\{\sigma(c(H))\}\cup N(\sigma)$.
Note that every core is peeled.

\begin{lemma}\label{grand}
Let $(H_0,\sigma_0)$ be a $k$-system of magnitude $d$ in a graph $G_0$.  Then there exists a reduction $(G,H,\sigma)$ of
$(G_0,H_0,\sigma_0)$ that is peeled, the $k$-system $(H,\sigma)$ has magnitude $\frac{d}{k(k+1)}$
and no vertex of $G$ other than $\sigma(c(H))$ has degree greater than $2k+1$.
\end{lemma}
\begin{proof}
Let $(G,H,\sigma)$ be a peeled reduction of $(G_0,H_0,\sigma_0)$, where the $k$-system $(H,\sigma)$ has magnitude $\frac{d}{k(k+1)}$,
with $|E(G)|$ as small as possible (which exists by Lemmas~\ref{reddeg} and \ref{paspl}).  Suppose that a vertex $v\in V(G)$ has
degree at least $2k+2$ and $v\neq \sigma(c(H))$.  By \icanonical{2} and the assumption that $(G,H,\sigma)$ is peeled, we have
$v\in M(\sigma)$.  Note that $v$ is joined with $\sigma(c(H))$ by at least $k+1$ edges.  Select an arbitrary edge $e\in E(H)$ such that $\sigma(e)$
contains $v$ and let $f_1$ and $f_2$ be the edges of $\sigma(e)$, where $f_1$ is incident with $\sigma(c(H))$.
Let $G'$ be the graph obtained from $G$ by lifting $f_1$ and $f_2$, creating a new edge $f$.  Let $\sigma'$ be obtained from $\sigma$
by letting $\sigma'(e)$ be the path consisting only of $f$.  Note that $(G',H,\sigma')$ is a reduction of $(G,H,\sigma)$.

We claim that $(H,\sigma')$ is a $k$-system in $G'$. It suffices to check that it satisfies the condition \igood{3}.
Let $K'$ be a minimal edge cut in $G'$ separating $\sigma'(c(H))$ from $\sigma'(x)$ for some $x\in R(H)$.
Let $C$ and $X$ be the vertex sets of the components of $G'-K'$, where $C$ contains $\sigma'(c(H))$.
Let $K$ be the set of edges between $C$ and $X$ in $G$.  If $K$ does not contain $f_1$, then $|K'|=|K|$, and thus
$|K'|\ge k$ by \igood{3} for $(G,H,\sigma)$.  If $K$ contains $f_1$, then $K$ also contains all edges parallel to $f_1$,
and these edges belong to $K'$ as well.  Since $v$ and $\sigma(c(H))$ are joined by at least $k+1$ edges, we have $|K'|\ge k$.
We conclude that $(G',H,\sigma')$ satisfies the condition \igood{3}.

Note that $(G',H,\sigma')$ is peeled, and thus it contradicts the minimality of $(G,H,\sigma)$.
\end{proof}

\begin{lemma}\label{lemma-bnddeg}
Let $k\ge 3$ be an integer and let $(H,\sigma)$ be a $k$-system in a graph $G$,
where $(G,H,\sigma)$ is peeled.  For each $v\in \sigma(R(H))$, there exists a set of $k$ pairwise edge-disjoint paths in $G$, each of length
at most $4k+2$, joining $\sigma(v)$ with $\sigma(c(H))$.
\end{lemma}
\begin{proof}
By \igood{3}, there exist $k$ pairwise edge-disjoint paths $Q_1$, \ldots, $Q_k$ between $\sigma(v)$ and $\sigma(c(H))$; let $Q=Q_1\cup\ldots\cup Q_k$
and let $S$ be the set of edges $e\in E(H)$ such $\sigma(e)\subseteq Q_i$ for some $i\in \{1,\ldots, k\}$.
Let us choose these paths so that $|E(Q)|$ is as small as possible, and subject to that $|S|$ is as large as possible.
Clearly, $Q$ contains exactly $k$ edges incident with $\sigma(c(H))$.

Let $e$ be an edge of $H$ such that $\sigma(e)$ shares at least one edge $f$ with $Q$, say $f\in E(Q_1)$.
Suppose that $f$ is not incident with $\sigma(c(H))$, and let $f'$ be the other edge of $\sigma(e)$.
If $f'$ were not in $E(Q)$, then we could change $Q_1$ to use $f'$ to enter $\sigma(c(H))$,
thus either decreasing $|E(Q)|$, or adding $e$ to $S$, contrary to the choice of the paths $Q_1$, \ldots, $Q_k$.
We conclude that if $\sigma(e)$ contains an edge of $Q$, then $Q\cap \sigma(e)$ contains an edge incident with $\sigma(c(H))$.
Consequently, there are at most $k$ such edges $e\in E(H)$.  Let $$W=\bigcup_{e\in E(H), E(\sigma(e))\cap E(Q)\neq\emptyset} V(\sigma(e))\setminus\{\sigma(c(H))\}.$$
Since $\sigma(e)$ has length at most two for each $e\in E(H)$, we have $|W|\le 2k$.

Consider the path $Q_i$ for some $i\in \{1,\ldots, k\}$, and let $v_0v_1\ldots v_{\ell}$ be its vertices in order, where
$v_0=\sigma(c(H))$ and $v_{\ell}=v$.  Let $Z$ be the set of vertices of $Q_i$ belonging to $\sigma(R(H))\cup M(\sigma)$.
Suppose that there exists a vertex $z\in Z\setminus W$.  Note that there exists an edge $e\in E(H)$ such that $z$ belongs to $\sigma(e)$,
and by the definition of $W$, we have $E(\sigma(e))\cap E(Q)=\emptyset$.  Therefore, we can change the path $Q_i$ to follow
$\sigma(e)$ from $z$ to $\sigma(c(H))$.  Since we have chosen the paths with $|E(Q)|$ as small as possible, we conclude that the distance from
$z$ to $\sigma(c(H))$ in $Q_i$ is at most two.  Consequently, $(Z\setminus W)\cap V(Q_i)\subseteq \{v_1,v_2\}$.
Since $(G,H,\sigma)$ is peeled and satisfies \icanonical{4}, every edge of $Q_i$ is incident with a vertex of $Z$, and thus
each edge of $Q_i-\{v_0,v_1,v_2\}$ is incident with a vertex of $Z\cap W$.  Since $|W|\le 2k$ and $v_{\ell}=v\in Z$, this implies that
$Q_i$ has length at most $4k+2$.
\end{proof}

\begin{proof}[Proof of Theorem~\ref{thm-main}.]
Let $N=(2k+1)^{8k+4}$ and $d=d(k)n=k^2(k+1)Nn$.  By Lemma~\ref{grand}, there exists a peeled reduction
$(G', H', \sigma')$ of $(G,H,\sigma)$, where $(H',\sigma')$ is a $k$-system of magnitude $d'=\frac{d}{k(k+1)}$
in $G'$ and no vertex of $G'$ other than $\sigma'(c(H'))$ has degree greater than $2k+1$.
Consequently, for each $v\in \sigma'(R(H'))$, there exist at most $N$ vertices at distance at most $8k+4$
from $v$ in $G'-\sigma'(c(H'))$.  Since $|R(H')|\ge d'/k\ge Nn$, we can greedily choose a set $U\subset \sigma'(R(H'))$
of size $n$ such the distance in $G'-\sigma'(c(H'))$ between any two vertices of $U$ is at least $8k+5$.
By Lemma~\ref{lemma-bnddeg}, we can for each $u\in U$ find a set $S_u$ of $k$ pairwise edge-disjoint paths in $G'$ joining $u$ with
$\sigma'(c(H'))$, each of length at most $4k+2$.
By the choice of $U$, the paths $\bigcup_{u\in U} S_u$ are pairwise edge-disjoint and intersect only in their endvertices.
This set of paths corresponds to a strong immersion of $S_{n,k}$ in $G'$ respecting $(H', \sigma')$.
We conclude that $G$ contains a strong immersion of $S_{n,k}$ respecting $(H,\sigma)$.
\end{proof}

\bibliographystyle{siam}
\bibliography{star}

\end{document}